\UseRawInputEncoding
\documentclass[12pt]{article}

\usepackage{dsfont}
\usepackage{amsfonts,amsmath,amsthm}
\usepackage{amssymb}
\usepackage{algorithm}
\usepackage{algpseudocode}
\usepackage{xcolor}
\usepackage{graphicx}
\usepackage{stmaryrd}        
\usepackage{multirow}

\usepackage[paper=a4paper,dvips,top=2cm,left=2cm,right=2cm,
foot=1cm,bottom=4cm]{geometry}

\newcommand{\ve}{{\bf e}}

\begin{document}
	\large
	
	\title{Positive Sem-idefinite and Sum of Squares Biquadratic Polynomials}
	\author{Chunfeng Cui\footnote{School of Mathematical Sciences, Beihang University, Beijing  100191, China.
			({\tt chunfengcui@buaa.edu.cn})}
		\and
		Liqun Qi\footnote{Jiangsu Provincial Scientific Research Center of Applied Mathematics, Nanjing 211189, China.
			Department of Applied Mathematics, The Hong Kong Polytechnic University, Hung Hom, Kowloon, Hong Kong.
			({\tt maqilq@polyu.edu.hk})}
		\and {and \
			Yi Xu\footnote{School of Mathematics, Southeast University, Nanjing  211189, China. Nanjing Center for Applied Mathematics, Nanjing 211135,  China. Jiangsu Provincial Scientific Research Center of Applied Mathematics, Nanjing 211189, China. ({\tt yi.xu1983@hotmail.com})}
		}
	}

	\date{\today}
	\maketitle
	
	\begin{abstract}
	Hilbert proved in 1888 that a positive semi-definite (PSD) homogeneous quartic polynomial of three variables  always can be expressed as the sum of squares (SOS) of three quadratic polynomials, and a psd homogeneous quartic polynomial of four variables may not be sos.  Only after {87} years, in {1975}, {Choi} gave the explicit expression of such a psd-not-sos (PNS) homogeneous quartic polynomial of four variables.   An $m \times n$ biquadratic polynomial is a homogeneous quartic polynomial of $m+n$ variables. 		In this {paper}, we show that an $m \times n$ biquadratic polynomial {can} be expressed as a tripartite homogeneous quartic polynomial of $m+n-1$ variables.  Therefore, {by Hilbert's theorem}, a $2 \times 2$ PSD biquadratic polynomial {can be expressed as the sum of squares of three quadratic polynomials.   This  improves the result of Calder\'{o}n in 1973, who proved that a $2 \times 2$ biquadratic polynomial can be expressed as the sum of squares of nine quadratic polynomials.}      Furthermore, we present a necessary and sufficient condition for {an $m \times n$}  psd biquadratic polynomial to be sos, and show that if such a polynomial is sos, then its sos rank is at most $mn$.
	{Then we give a constructive proof of
		the sos form of a $2 \times 2$ psd biquadratic polynomial in three cases.}
		

		\medskip

		\medskip

		
	\end{abstract}

	\renewcommand{\Re}{\mathds{R}}
	\newcommand{\rank}{\mathrm{rank}}
	\newcommand{\X}{\mathcal{X}}
	\newcommand{\A}{\mathcal{A}}
	\newcommand{\I}{\mathcal{I}}
	\newcommand{\B}{\mathcal{B}}
	\newcommand{\PP}{\mathcal{P}}
	\newcommand{\C}{\mathcal{C}}
	\newcommand{\D}{\mathcal{D}}
	\newcommand{\LL}{\mathcal{L}}
	\newcommand{\OO}{\mathcal{O}}
	\newcommand{\e}{\mathbf{e}}
	\newcommand{\0}{\mathbf{0}}
	\newcommand{\1}{\mathbf{1}}
	\newcommand{\dd}{\mathbf{d}}
	\newcommand{\ii}{\mathbf{i}}
	\newcommand{\jj}{\mathbf{j}}
	\newcommand{\kk}{\mathbf{k}}
	\newcommand{\va}{\mathbf{a}}
	\newcommand{\vb}{\mathbf{b}}
	\newcommand{\vc}{\mathbf{c}}
	\newcommand{\vq}{\mathbf{q}}
	\newcommand{\vg}{\mathbf{g}}
	\newcommand{\pr}{\vec{r}}
	\newcommand{\pc}{\vec{c}}
	\newcommand{\ps}{\vec{s}}
	\newcommand{\pt}{\vec{t}}
	\newcommand{\pu}{\vec{u}}
	\newcommand{\pv}{\vec{v}}
	\newcommand{\pn}{\vec{n}}
	\newcommand{\pp}{\vec{p}}
	\newcommand{\pq}{\vec{q}}
	\newcommand{\pl}{\vec{l}}
	\newcommand{\vt}{\rm{vec}}
	\newcommand{\x}{\mathbf{x}}
	\newcommand{\vx}{\mathbf{x}}
	\newcommand{\vy}{\mathbf{y}}
	\newcommand{\vu}{\mathbf{u}}
	\newcommand{\vv}{\mathbf{v}}
	\newcommand{\y}{\mathbf{y}}
	\newcommand{\vz}{\mathbf{z}}
	\newcommand{\T}{\top}
	\newcommand{\R}{\mathcal{R}}
	\newcommand{\Q}{\mathcal{Q}}
	\newcommand{\TT}{\mathcal{T}}
	\newcommand{\Sc}{\mathcal{S}}
	\newcommand{\N}{\mathbb{N}}	
	
	\newtheorem{Thm}{Theorem}[section]
	\newtheorem{Def}[Thm]{Definition}
	\newtheorem{Ass}[Thm]{Assumption}
	\newtheorem{Lem}[Thm]{Lemma}
	\newtheorem{Prop}[Thm]{Proposition}
	\newtheorem{Cor}[Thm]{Corollary}
	\newtheorem{example}[Thm]{Example}
	\newtheorem{remark}[Thm]{Remark}

	\section{Introduction}
	
	In	1888, when he was 26, David Hilbert~\cite{Hi88} pointed out that for homogeneous polynomials, only in the following three cases, a~positive semi-definite ({PSD}) polynomial definite is a sum of squares ({SOS}) polynomial: (i) $N=2$; (ii) $M=2$; (iii) $M=4$; and $N=3$, where $M$ is the degree of the polynomial and {$N$} is the number of variables.   Hilbert proved that in all the other possible combinations of $M=2K$ and $N$, there are {PSD} non-{SOS} ({PNS}) homogeneous polynomials.    {Furthermore, Hilbert established that every  PSD  homogeneous quartic polynomial in three variables admits a representation as an SOS of three quadratic forms. His original proof of this result, demonstrating the equivalence between PSD and SOS for ternary quartics, combines striking conciseness with pioneering topological methods that were remarkably advanced for their era, while mathematically elegant, the~proof's innovative approach proved difficult  to fully comprehend, and~even modern readers must carefully work through several non-obvious details. Recent  literature has witnessed both comprehensive reconstructions of Hilbert's original proof and the emergence of novel alternative proofs. For~instance, in~2012, Pfister and Scheiderer~\cite{PS12} presented a completely	new approach that uses only elementary~techniques.
		
		While Hilbert established in 1888 that PSD  polynomials coincide with  SOS  in only three specific cases, the~first explicit example of a  PNS  homogeneous polynomial was not constructed until 77 years later, in~1965, when Motzkin~\cite{Mo65} gave the first  PNS homogeneous polynomial:}
	$$f_M(\x) = x_3^6+x_1^2x_2^4+x_1^4x_2^2 - 3x_1^2x_2^2x_3^2,$$
	{where}   $M=6$ and $N=3$.  Thus, the~ remaining case with small $M$ and $N$ is that $M = N = 4$.  In~1977, Choi and Lam~\cite{CL77} gave such a {PNS} polynomial:
	\begin{equation} \label{fcl}
		f_{CL}(\x) = x_1^2x_2^2 + x_1^2x_3^2 + x_2^2x_3^2 - 4x_1x_2x_3x_4 + x_4^4.
	\end{equation}
	Also,	see~\cite{Ch07, Re78} for {PNS} homogeneous {polynomials.}
	In 2017, Goel~et~al.~\cite{GKR17} {showed} that a  symmetric {PSD} polynomial with $M = N = 4$ is {SOS.}
	This raises a natural question: Do analogous results hold for nonsymmetric quartic forms under additional structural constraints?

	In this paper, we study the {PSD-SOS} problem of a special class of homogeneous quartic polynomials: biquadratic polynomials.   In~general, an~$m \times n$ biquadratic polynomial {can be expressed as follows:}
	\begin{equation} \label{BQ}
		f(\hat \x, \hat \y) = \sum_{i, j=1}^m \sum_{k, l=1}^n a_{ijkl}x_ix_jy_ky_l,
	\end{equation}
	where $\hat \x = (x_1, \dots, x_m)^\top$ and $\hat \y = (y_1, \dots, y_n)^\top$.
	{
		Biquadratic polynomials and biquadratic tensors arise in solid mechanics, statistics,  quantum physics, spectral graph theory, and~polynomial theory~\cite{QC25}. Biquadratic SOS decompositions enable efficient solutions to non-convex optimization problems via SDP relaxations. In~the analysis of nonlinear dynamical systems, stability can be rigorously verified through the construction of Lyapunov functions that admit an SOS  decomposition.}

	Without loss of generality, we may assume that $m \ge n \ge 2$.
	Thus, an~$m \times n$ biquadratic polynomial is a quartic homogeneous polynomial of $m+n$ variables, i.e.,~$M= 4$ and $N = m+n$.   However, we will show that for the {PSD-SOS} problem, this value of $m+n$ can be reduced by one.   In Section~\ref{sec2}, we show that for an $m \times n$ biquadratic polynomial $f$, there exists a special quartic homogeneous polynomial $h$, which we call a tripartite quartic polynomial, with~$N = m+n-1$ variables, such that $f$ is {PSD} if and only if $h$ is {PSD,} and $f$ is {SOS} if and only if $h$ is {SOS.}   Then, according to~Hilbert's result, we immediately conclude that a $2 \times 2$ {PSD} biquadratic polynomial {can be expressed as the sum of squares of three quadratic polynomials.   This greatly improves the result of Calder\'{o}n \cite{Ca73} in 1973, who proved that a $2 \times 2$ biquadratic polynomial can be expressed as the sum of squares of nine quadratic polynomials.}
	{Furthermore,} we present a necessary condition for a tripartite quartic polynomial to be {PSD}, and~a necessary and sufficient condition for it to be {PSD} under the condition that the coefficient of the quartic power term of the flexible variable is~zero.
	
	In Section~\ref{sec3}, we give a sufficient and necessary condition for an $m \times n$ {PSD} biquadratic polynomial to be {SOS.}  We show that if an $m \times n$ {PSD} biquadratic polynomial is {SOS,} than its {SOS} rank is at most $mn$.
	
	For a $2 \times 2$ {PSD} biquadratic polynomial, we may explicitly construct its {SOS} representation  in some cases.
	A $2 \times 2$ biquadratic polynomial $f$ has four half-cross terms and one full-cross term.
	{In Section~\ref{sec4},},  we show that for a $2 \times 2$ {PSD} biquadratic polynomial $f$, there is another $2 \times 2$ {PSD} biquadratic polynomial $h$, which has at most two half-cross terms, such that if $h$ has an {SOS} form, then the {SOS} form of $f$ can be {explicitly constructed.}   Then we give a constructive proof of
	the {SOS} form of a $2 \times 2$ {PSD} biquadratic polynomial $f$, if~it has either at most one half-cross term, or~two half-cross terms but no full-cross~term.
	
	Some final remarks and open questions are made and raised in Section~\ref{sec5}.

	\section{Biquadratic Polynomials and Tripartite Quartic~Polynomials\label{sec2}}
	
	The following theorem is one main {result} of this~paper.
	
	\begin{Thm} \label{BQ-TRI}
		Suppose that we have an $m \times n$ biquadratic polynomial $f$, expressed by~(\ref{BQ}), with~\mbox{$m \ge n \ge 2$}.  Then there is a tripartite homogeneous quartic polynomial $h(\x, \y, z)$ of $m+n-1$ variables, where $\x = (x_1, \dots, x_{m-1})^\top$ and $\y = (y_1, \dots, y_{n-1})^\top$, such that $f$ is {PSD} if and only if $h$ is {PSD,} and $f$ is {SOS} if and only if $h$ is {SOS.}   In particular, a~$2 \times 2$ {PSD} biquadratic polynomial is always {SOS} of three quadratic polynomials.
	\end{Thm}
	\begin{proof}  Let $x_m = 1$ and $y_n = 1$  for $j = 1, \dots, n-1$.  Then we have a non-homogeneous quartic polynomial $g(\x, \y)$, of~ $m+n-2$ variables $x_1, \dots, x_{m-1}$, $y_1, \dots, y_{n-1}$, such that $f$ is {PSD} if and only if $g$ is {PSD,} and $f$ is {SOS} if and only if $g$ is {SOS.}
		
		Now, we may convert $g$ to a homogeneous quartic polynomial $h$ of $m+n-1$ variables, by~adding an additional $z$ in the position of $x_m$ and $y_n$, i.e.,~$h$ is obtained by replacing $x_m$ and $y_n$ by $z$.    Then $g$ is {PSD} if and only if $h$ is {PSD,} and $g$ is {SOS} if and only if $h$ is {SOS.}  The final conclusion now follows from the 1888 Hilbert result~\cite{Hi88}. The~proof is complete.
	\end{proof}

	We call $h(\x, \y, z)$ a tripartite homogeneous quartic polynomial, as~it has the following property: in any term of $h$, the~total degree of $x$-variables is at most $2$, and~the total degree of $y$-variables is also at most $2$.  We call $z$ the flexible variable.  Only the flexible variable can have degree four in $h$.

	A tripartite quartic homogeneous polynomial $h(\x, \y, z)$ can be written as follows:
	\begin{equation} \label{TRI}
		h(\x, \y, z) = h_0{z^4} + h_1(\x, \y){z^3} + h_2(\x, \y){z^2} + h_3(\x, \y)z + h_4(\x, \y),
	\end{equation}
	where $h_0$ is a constant coefficient; $h_4(\x, \y)$ is an $(m-1) \times (n-1)$ biquadratic polynomial; $h_1(\x, \y)$ is a linear homogeneous polynomial of $(\x, \y)$; $h_2(\x, \y)$ is the sum of a quadratic homogeneous polynomial of $\x$, a~quadratic homogeneous polynomial of $\y$, and~a bilinear polynomial of $(\x, \y)$; and $h_3(\x, \y)$ is the sum of a linear quadratic polynomial and a quadratic-linear polynomial of $(\x, \y)$.    We have the following~theorem.
	
	\begin{Thm} \label{TRI-BQ}
		Suppose that we have a tripartite quartic polynomial $h(\x, \y, z)$, expressed by~(\ref{TRI}), where $\x \in \Re^{m-1}, \y \in \Re^{n-1}$, $z \in \Re$ and $m \ge n \ge 2$.  Then there is an $m \times n$ biquadratic polynomial $f(\hat \x, \hat \y)$, where $\hat \x = (x_1, \dots, x_m)^\top$ and $\hat \y = (y_1, \dots, y_n)^\top$, such that $h$ is {PSD} if and only if $f$ is {PSD,} and $h$ is {SOS} if and only if $f$ is {SOS.}
	\end{Thm}
	\begin{proof}  In~(\ref{TRI}), let $z = 1$.  Then we have a non-homogeneous quartic polynomial $g(\x, \y)$, of~$m+n-2$ variables $x_1, \dots, x_{m-1}, y_1, \dots, y_{n-1}$, such that $h$ is {PSD} if and only if $g$ is {PSD,} and $h$ is {SOS} if and only if $g$ is {SOS.}
		
		Now, we may convert $g$ to an $m \times n$ biquadratic polynomial $f(\hat \x, \hat \y)$, by~adding an additional variable $x_m$ or $x_m^2$, if~needed, to~terms of $g$ to make the total degrees of $x$-variables of all terms of $g$ as $2$, and~an additional variable $y_n$ or $y_n^2$, if~needed, to~terms of $g$ to make the total degrees of {$y$-variables} of all terms of $g$ as $2$.
		Then $g$ is {PSD} if and only if $f$ is {PSD,} and $g$ is {SOS} if and only if $f$ is {SOS.}  The proof is complete.
	\end{proof}
	
	Thus, the~{PSD-SOS} problem of $m \times n$ biquadratic polynomials is equivalent to the {PSD-SOS} problem of tripartite quartic polynomials of $m+n-1$ variables.
	
	To consider the {PSD-SOS} problem of tripartite quartic polynomials, we have the following~theorem.
	
	\begin{Thm}
		Suppose that $h(\x, \y, z)$ is a tripartite quartic polynomial, expressed by~(\ref{TRI}).  If~$h(\x, \y, z)$ is {PSD,} then $h_0 \ge 0$ and $h_4(\x, \y)$ is {PSD.}
		
		Furthermore, if $h_0 = 0$, then the tripartite quartic polynomial $h(\x, \y, z)$ is {PSD} if and only if (i) $h_1(\x, \y) \equiv 0$, (ii) $h_2(\x, \y)$, $h_4(\x, \y)$ and $4h_2(\x, \y)h_4(\x, \y) - h_3(\x, \y)^2$ are {PSD.} 
	\end{Thm}
	\begin{proof}
		Assume that $h$ is {PSD.}   Let $x$-variables and $y$-variables be zero and $z=1$.  Since $h$ is {PSD,} we have $h_0 \ge 0$.  Let $z = 0$.  Since $h$ is {PSD,} $h_4(\x, \y)$ should be {PSD.}
		
		Suppose now that $h_0 = 0$.
		
		First, we assume that $h$ is {PSD.}   Since $h_1(\x, \y)$ is a linear homogeneous polynomial of $(\x, \y)$, if~$h_1(\x, \y) \not \equiv 0$, we may choose adequate values of $(\x, \y)$ such that $h_1$ takes a negative value.  Let $z$ goes to infinity.  Then, we have $h(\x, \y, z) < 0$, contradicting with the assumption that $f$ is {PSD.}  Thus, $h_1(\x, \y) \equiv 0$.  Similarly, if~$h_0 = 0$, we can show that $h_2(\x, \y)$ is {PSD.}     Furthermore, if~$4h_2(\x, \y)h_4(\x, \y) - h_3(\x, \y)^2$ is not {PSD,} then there are some values $(\x^*, \y^*)$ such that the function $4h_2(\x^*, \y^*)h_4(\x^*, \y^*) - h_3(\x^*, \y^*)^2 < 0$.  Then, $h(\x^*, \y^*, z)$ is an indeterminate quadratic polynomial of $z$.  There is a value $z^*$ such that $h(\x^*, \y^*, z^*) < 0$, contradicting to the assumption that $h$ is {PSD.}   Thus, $4h_2(\x, \y)h_4(\x, \y) - h_3(\x, \y)^2$ must be~{PSD.}
		
		On the other hand, assume that conditions (i) and (ii) hold.   Then for any value $(\x^*, \y^*)$, $h(\x^*, \y^*, z)$ is a {PSD} quadratic polynomial of $z$.  Then we always have $h(\x, \y, z) \ge 0$, i.e.,~$h$ is {PSD.}
		The proof is complete.
	\end{proof}
	
	Thus, to~consider the {PSD-SOS} problem of tripartite quartic polynomials, we only need to investigate two types of tripartite quartic polynomials.  One type is that $h_0= 1$ and $h_4(\x, \y)$ is {PSD} in~(\ref{TRI}).  We call such a tripartite quartic polynomial a nondegenerated tripartite quartic polynomial.   Another type is that $h_0 = 0, h_1(\x, \y) \equiv 0$, $h_2(\x, \y)$, $h_4(\x, \y)$ and
	$4h_2(\x, \y)h_4(\x, \y) - h_3(\x, \y)^2$
	are {PSD} in~(\ref{TRI}). 
	We call such a tripartite quartic polynomial a degenerated tripartite quartic~polynomial.

	In fact, $h_0 = 0$ holds if and only if {$a_{mnmn} = 0$.}
	Moreover, in~the construction of the function $g$ in Theorem~\ref{BQ-TRI}, the~choice of indices $i \in \{1, \dots, m\}$ and $j \in \{1, \dots, n\}$ is arbitrary. Consequently, $h$ is degenerated whenever there exist $i$ and $j$ such that {$a_{ijij} = 0$.}

	Note that a degenerated tripartite quartic polynomial is {PSD.}  A natural question is as follows:  Is a degenerated tripartite quartic polynomial always {SOS?}

	\section{SOS Biquadratic~Polynomials\label{sec3}}
	
	{Let $\A=(a_{ijkl})$ be an $m\times n$ symmetric   biquadratic tensor~\cite{QDH09} that satisfies $a_{ijkl}=a_{klij}$.}
	We can rewrite the $m\times n$ biquadratic polynomial as follows:
	\begin{eqnarray}\label{equ:f=zMz}
		f(\vx,\vy) = \sum_{i,k=1}^m\sum_{j,l=1}^n a_{ijkl}x_{i}y_{j}x_{k}y_{l} =\vz^T ({B}+P(\Gamma))\vz:=\vz^T M(\Gamma)\vz,
	\end{eqnarray}
	where $\vz=\vx\otimes \vy\in \Re^{mn}$, {the matrix $B=(b_{st})\in \Re^{mn\times mn}$ is obtained by flattening the fourth-order tensor $\A$ into a two-dimensional representation, where the first two dimensions of the tensor are arranged as matrix rows and the last two dimensions as matrix columns.~Namely,}
	\begin{eqnarray*}
		{{B}}=\begin{bmatrix}
			{B}_{11} & {B}_{12} & \cdots & {B}_{1m} \\
			{B}_{21} & {B}_{22} & \cdots & {B}_{2m} \\
			\cdots  & \cdots   & \cdots & \cdots  \\
			{B}_{m1} & {B}_{m2} & \cdots & {B}_{mm} \\
		\end{bmatrix}, \ \    {B_{ik}} = \begin{bmatrix}
			a_{i1k1} & a_{i1k2}  & \cdots & a_{i1kn}  \\
			a_{i2k1} & a_{i2k2}  & \cdots & a_{i2kn} \\
			\cdots  & \cdots   & \cdots & \cdots  \\
			a_{ink1} & a_{ink2}  & \cdots & a_{inkn} \\
		\end{bmatrix}.  	
	\end{eqnarray*}
	{Here}, 		each element is given by  $b_{st} = a_{ijkl}=a_{klij}$, where  $s=(i-1)n+j$ and $t=(k-1)n+l$.
		
	{While}
		the coefficients of diagonal {(where $i=k$ and $j=l$)} and half-cross  {(where $i=k$ or $j=l$)} 
		terms 	  are uniquely determined by the symmetry of $B$, the~full-cross terms (where $i\neq k$ and $j\neq l$)  exhibit non-uniqueness, as~evidenced by the identity:
		\begin{eqnarray*}
			&&\left(a_{ijkl}+a_{klij}+a_{kjil}+a_{ijkl}\right) x_{i}y_{j}x_{k}y_{l} \\
			&=& 2\left(a_{ijkl}+\gamma \right)x_{i}y_{j}x_{k}y_{l}  +2\left(a_{kjil}-\gamma \right) x_{i}y_{j}x_{k}y_{l} 	
		\end{eqnarray*}
		for any $\gamma \in \Re$.
		Specifically, let us define the index mappings   $s_1=(i-1)n+j$, $t_1=(k-1)n+l$, and~$s_2=(k-1)n+j$, $t_2=(i-1)n+l$. The~quadratic form remains invariant under the transformation  $b_{s_1t_1}\rightarrow b_{s_1t_1}+\gamma$ and   $b_{s_2t_2}\rightarrow b_{s_2t_2}-\gamma$.
		Therefore, we use
		$\Gamma=(\gamma_{{t_it_j}}) \in \Re^{\binom{m}{2}\times \binom{n}{2}}$ to represent  a free parameter that captures the degrees of freedom in expressing the full cross terms, and~construct $P(\Gamma) = (P_{ij}(\Gamma))\in\Re^{mn\times mn}$ as follows:
	\begin{eqnarray*}
		P_{ij}(\Gamma) = \begin{cases}
			\gamma_{{t_it_j}}, & \text{ if } \exists i_1<i_2,j_1<j_2 \text{ s.t. } i=n_{i_1j_1}, j = n_{i_2j_2}, {t_i}=t_{i_1i_2}, {t_j}=t_{j_1j_2};\\
			\gamma_{{t_it_j}}, & \text{ if } \exists i_1>i_2,j_1>j_2 \text{ s.t. } i=n_{i_1j_1}, j = n_{i_2j_2}, {t_i}=t_{i_1i_2}, {t_j}=t_{j_1j_2};\\
			-\gamma_{{t_it_j}}, & \text{ if } \exists i_1<i_2,j_1>j_2 \text{ s.t. } i=n_{i_1j_1}, j = n_{i_2j_2}, {t_i}=t_{i_1i_2}, {t_j}=t_{j_1j_2};\\
			-\gamma_{{t_it_j}}, & \text{ if } \exists i_1>i_2,j_1<j_2 \text{ s.t. } i=n_{i_1j_1}, j = n_{i_2j_2}, {t_i}=t_{i_1i_2}, {t_j}=t_{j_1j_2};\\
			0, & \text{otherwise}.
		\end{cases}
	\end{eqnarray*}
	{Here},  $\binom{m}{2}=\frac{m(m-1)}{2}$,  $n_{i_1j_1}=(i_1-1)n+j_1$,  $t_{i_1i_2} = \frac{(i_1-1)(i_1-2)}{2}+i_2$ if $i_1>i_2$ and $t_{i_1i_2} = \frac{(i_2-1)(i_2-2)}{2}+i_1$ if $i_1<i_2$.

	{For} 	instance, when $m=n=2$, it holds that $\vz=[x_1y_1, x_1y_2, x_2y_1, x_2y_2]^{\top}$,
	\begin{equation*}
		B=\begin{bmatrix}
			a_{1111} & a_{1112}& a_{1121} & a_{1122}\\
			a_{1211} & a_{1212} &  a_{1221}  &  a_{1222}
			\\
			a_{2111}  &  a_{2112}  & a_{2121} &  a_{2122}
			\\
			a_{2211}  & a_{2212}  & a_{2221}  & a_{2222}
		\end{bmatrix} \text{ and } P(\gamma)=\begin{bmatrix}
			& & & \gamma \\
			& & -\gamma & \\
			& -\gamma & & \\
			\gamma& & &
		\end{bmatrix}.
	\end{equation*}

	Then we have the following~theorem.

	\begin{Thm}\label{thm:sos}
		Suppose that the $m\times n$ biquadratic polynomial $f$ expressed by \eqref{equ:f=zMz} is {PSD.}
		Then the following three statements are~equivalent.
		\begin{enumerate}
			\item[(i)] 	$f$ is {SOS,} with its {SOS} rank at most $mn$, and~the coefficients can be orthogonal to each other;
			\item[(ii)] there exists $\Gamma=(\gamma_{{t_it_j}}) \in \Re^{\binom{m}{2}\times \binom{n}{2}}$ such that  $M(\Gamma)$ is {PSD;}
			\item[(iii)] the following condition holds,
			\begin{equation}
				\max_{\Gamma \in \Re^{\binom{m}{2}\times \binom{n}{2}}}\min_{\vz^\top\vz = 1} \vz^{\top}M(\Gamma) \vz \ge 0.
			\end{equation}
		\end{enumerate}
	\end{Thm}
	\begin{proof}
		$(i)\Longrightarrow (ii)$  If $f$ is {SOS,} then there exists $T>0$ such that $f=\sum_{t=1}^Tf_t^2(\vx,\vy)$.
		By Theorem 1 in~\cite{Re78}, we obtain that ${\rm supp}(f_t)\subseteq {\frac12{\rm New}(f)}$, $\forall\, t=1,\dots,T$.    Here, {\rm New}$(f)$ is the Newton polytope of $f$.
		Therefore, we deduce that each $f_t$ must  be of the linear form \mbox{$f_t=\vc_t^{\top}\vz$}, where $\vc_t$ is a coefficient vector and $\vz=\vx\otimes \vy$.
		Let $C=[\vc_1,\dots,\vc_T]\in\Re^{mn\times T}$ denote the coefficient matrix. This immediately gives the quadratic representation  \mbox{$f = \vz^{\top} CC^{\top}\vz$}.
		Consequently, there exists $\Gamma$ such that  $M(\Gamma)=CC^{\top}$ is {PSD.}

		$(ii)\Longrightarrow (i)$ Since $M(\Gamma)$ is {PSD,} it admits a Cholesky decomposition. Specifically, there exists a matrix $C\in\Re^{mn\times mn}$ such that  ${M(\Gamma)}=CC^{\top}$.  Therefore,
		$f=\sum_{t=1}^{mn}(\vc_t^{\top}\vz)^2$ is an {SOS}  of tetranomials, with~  {SOS} rank at most $mn$, and~the {coefficients}
		are orthogonal to each~other.
		
		$(ii)\Longleftrightarrow (iii)$ Since $M(\Gamma)$ is symmetric,   its positive semidefiniteness is equivalent to the nonnegativity of all its eigenvalues. In~particular, $M(\Gamma)$ is {PSD} if and only if its smallest eigenvalue $\lambda_{\min} = \min_{\vz^\top\vz = 1} \vz^{\top}M(\Gamma) \vz$   is~nonnegative.
		
		This completes the proof.
	\end{proof}

	{
		Next, we use the  $2\times 2$ example on page 358 of Qi~et~al.~\cite{QDH09} to demonstrate that $B$ is not PSD, yet there exists $\gamma$ such that $f(\vx,\vy)$ is SOS.
		\begin{example}
			Given a  $2\times 2$ biquadratic tensor $\A$   with the following components:  $a_{1111}=1$, $a_{1112}=2$, $a_{1122}=4$, $a_{1212}=12$, $a_{2121}=12$, $a_{1222}=1$, $a_{1121}=2$, $a_{2122}=1$, and~$a_{2222}=2$. The~associated biquadratic polynomial can be expressed as follows:
			\begin{eqnarray*}
				f(\vx,\vy)&=&x_1^2y_1^2+4x_1^2y_1y_2+12x_1^2y_2^2+4x_1x_2y_1^2+16x_1x_2y_1y_2\\
				&&+2x_1x_2y_2^2+12x_2^2y_1^2+2x_2^2y_1y_2+2x_2^2y_2^2.
			\end{eqnarray*}
		{Then} we can obtain
			\begin{equation*}
				B=\begin{bmatrix}
					a_{1111} & a_{1112}& a_{1121} & a_{1122}\\
					a_{1211} & a_{1212} &  a_{1221}  &  a_{1222}
					\\
					a_{2111}  &  a_{2112}  & a_{212} &  a_{2122}
					\\
					a_{2211}  & a_{2212}  & a_{2221}  & a_{2222}
				\end{bmatrix}=\begin{bmatrix}
					1 & 2 & 2 & 4 \\
					2 & 12 & 4 & 1 \\
					2 & 4 & 12 & 1\\
					4 & 1 & 1 & 2
				\end{bmatrix}.
			\end{equation*}

			{The} 	matrix $B$ fails to be PSD, as~evidenced by its smallest eigenvalue of  $-2.6110$.
			However, by~selecting  $\gamma=-3$, we obtain a modified matrix $M(\gamma)$ that is PSD, with~its smallest eigenvalue now being $M(\gamma)$ is $0.2069$. This positive semidefiniteness is further confirmed by the existence of the following Cholesky decomposition
			\begin{equation*}
				M(\gamma) = \begin{bmatrix}
					1 & 2 & 2 & 1 \\
					2 & 12 & 7 & 1 \\
					2 & 7 & 12 & 1\\
					1 & 1 & 1 & 2
				\end{bmatrix}=CC^{\top}, \text{ with } C=
				\begin{bmatrix}
					1 &  0 &  0  &  0 \\
					2  &  2.8284  &  0  &   0 \\
					2  & 1.0607  & 2.6220  & 0 \\
					1  & -0.3536  & -0.2384  &  0.9045
				\end{bmatrix}
			\end{equation*}
			{Consequently},  $f(\vx,\vy)$ admits the following SOS expression
			\begin{eqnarray*}
				f(\vx,vy) &=& (x_1y_1+2x_1y_2+2x_2y_1+x_2y_2)^2+(2.8284x_1y_2+1.0607x_2y_1-0.3536x_2y_2)^2\\
				&& +(2.6220x_2y_1-0.2384x_2y_2)^2 +0.9045x_2y_2^2.
			\end{eqnarray*}
		\end{example}

	}
	
	Theorem~\ref{thm:sos} shows that if an $m \times n$ {PSD} biquadratic polynomial is {SOS,} then its {SOS} rank is at most $mn$, and~the {coefficients}
	can be orthogonal to each~other.

	It may be insightful to examine   $f$ in relation to  the  {PSDness} of $M(\Gamma)$.
	Define
	\begin{equation}
		S_{mn}=\{\vz:\vz =  \vx\otimes \vy, \vx\in\Re^m,\vy\in \Re^n\}.
	\end{equation}
	{Suppose} that $f$ is {PSD,} then there exists $\Gamma$ such that
	$\vz^{\top} M(\Gamma)\vz\ge 0$ for all $\vz\in S_{mn}$.
	{Denote $\ve_i(mn)$ as the $i$-th standard basis in $\Re^{mn}$.}
	It follows from $\ve_1(mn),\dots,\ve_{mn}(mn)\in S_{mn}$ that ${\rm dim}(S_{mn}) = mn$.   However,  $S_{mn}$ is not a convex set.
	Consequently, the~{PSD} property of $f$ does not necessarily guarantee that $f$ is {SOS}, in~general.

	\section{\boldmath$2 \times 2$ Biquadratic~Polynomials\label{sec4}}
	
	Given a   $2 \times 2$ biquadratic polynomials as follows:
	\begin{align}
		\nonumber	f(\x, \y) = & a_{11}x_1^2y_1^2 + a_{12}x_1^2y_2^2+a_{21}x_2^2y_1^2+a_{22}x_2^2y_2^2+bx_1x_2y_1y_2\\
		&+c_{x1}x_1^2y_1y_2+c_{x2}x_2^2y_1y_2+c_{y1}x_1x_2y_1^2+c_{y2}x_1x_2y_2^2, \label{equ:2x2_BQ}
	\end{align}
	where the first four terms are called {{diagonal}
		terms}, the~fifth term is called the {{full}-cross term}, and~the last four terms are called {{half}-cross terms}.   In Section~\ref{sec2}, we have shown that a $2 \times 2$ {PSD} biquadratic polynomial is always {SOS.}   In this section, we present some {SOS} forms of $2 \times 2$ {PSD} biquadratic~polynomials.
	
	We first  show that a   $2\times 2$ {PSD} biquadratic polynomial $f$ expressed by \eqref{equ:2x2_BQ}	can be reformulated equivalently to another {PSD} $2\times 2$ biquadratic polynomial with two neighbor half-cross terms and one full-cross~terms.
	
	\begin{Prop}\label{prop:3.1}
		Suppose that the $2\times 2$ biquadratic polynomial $f$ expressed by \eqref{equ:2x2_BQ} is {PSD.} Then
		$a_{11}, a_{12}, a_{21}, a_{22} \ge 0$, and~		$$4a_{11}a_{12} \ge c_{x1}^2,\ 4a_{11}a_{21} \ge c_{y1}^2,\ 4a_{12}a_{22} \ge c_{y2}^2,\ 4a_{21}a_{22} \ge c_{x2}^2.$$
		{Furthermore}, there is a $2\times 2$ {PSD} biquadratic polynomial $g$ with only three half-cross terms:
		\begin{align}
			\nonumber	g(\x, \y) = & \bar a_{11}x_1^2y_1^2 + \bar a_{12}x_1^2y_2^2+\bar a_{21}x_2^2y_1^2+\bar a_{22}x_2^2y_2^2+\bar bx_1x_2y_1y_2\\
			&+\bar c_{x1}x_1^2y_1y_2+\bar c_{y1}x_1x_2y_1^2+\bar c_{y2}x_1x_2y_2^2, \label{functiong}
		\end{align}
		such that if $g$ has an {SOS} form, then the {SOS} form of $f$ can be resulted.
	\end{Prop}
	\begin{proof}
		Let $x_1=y_1=1$ and $x_2=y_2=0$.  Then we obtain that $a_{11} \ge 0$.  Similarly, we derive that  $a_{12}, a_{21}, a_{22} \ge 0$.   Next, setting $x_1 = 1$ and $x_2=0$ leads to the inequality  $4a_{11}a_{12} \ge c_{x1}^2$.  Similar arguments establish the remaining conditions:
		$4a_{11}a_{21} \ge c_{y1}^2$, $4a_{12}a_{22} \ge c_{y2}^2$ and $4a_{21}a_{22} \ge c_{x2}^2$.
		
		
		Suppose $a_{21} = 0$. It follows from  $4a_{21}a_{22} \ge c_{x2}^2$ that  $c_{x2} = 0$.   Then $f$ has only three half-cross terms~itself.
		
		{Now suppose} that $a_{21} > 0$.  Let $\bar x_1 = x_1$, $\bar x_2 = x_2$, $\bar y_2 = y_2$,
		$$\bar y_1 = y_1 + {c_{x2} \over 2a_{21}}y_2$$
		and $g(\bar \x, \bar \y) \equiv f(\x, \y)$. Then $g$ has only three half-cross terms
		as follows:\vspace{-3pt}
		\begin{align*}
			g(\bar \x, \bar \y) = & f(\x, \y)\\
			= & a_{11}\bar x_1^2\bar y_1^2 + \left(\frac{a_{11}c_{x2}^2}{4a_{21}^2}+a_{12}-\frac{c_{x1}c_{x2}}{2a_{21}}\right)\bar x_1^2\bar y_2^2+{a_{21}}\bar x_2^2\bar y_1^2+\left(a_{22}-\frac{c_{x2}^2}{4a_{21}}\right)\bar x_2^2\bar y_2^2\\
			&+\left(b-\frac{c_{y1}c_{x2}}{a_{21}}\right)\bar x_1\bar x_2\bar y_1\bar y_2+\left(c_{x1}-\frac{a_{11}c_{x2}}{2a_{21}}\right)\bar x_1^2\bar y_1\bar~y_2\\
			&+c_{y1}\bar x_1\bar x_2\bar y_1^2+\left(c_{y2} +\frac{c_{y1}c_{x2}^2}{4a_{21}^2}-\frac{bc_{x2}}{2a_{21}}\right)\bar x_1\bar x_2\bar y_2^2,
		\end{align*}
		which is in  the form of~(\ref{functiong}).
		This completes the proof.
	\end{proof}
	
	\begin{Prop}\label{prop:3.2}
		Suppose that the $2\times 2$ biquadratic polynomial $g$ expressed by \eqref{functiong} is {PSD.} Then, there is a $2\times 2$ {PSD} biquadratic polynomial $h$ containing only two half-cross terms are in neighbor positions  as follows:
		\begin{align}
			\nonumber	h(\x, \y) = & \hat a_{11}x_1^2y_1^2 + \hat a_{12}x_1^2y_2^2+\hat a_{21}x_2^2y_1^2+\hat a_{22}x_2^2y_2^2+\hat bx_1x_2y_1y_2\\
			&+\hat c_{x1}x_1^2y_1y_2+\hat c_{y1}x_1x_2y_1^2, \label{functionh}
		\end{align}
		such that if $h$ has an {SOS} form, then the {SOS} form of $g$ can be resulted.
	\end{Prop}
	\begin{proof}
		
		Suppose ${\bar a_{22}} = 0$.  It follows from  $4{\bar a_{12}\bar a_{22}} \ge {\bar c_{y2}}^2$ that  ${\bar c_{y2}} = 0$.   Then $g$ itself has only two half-cross terms, which are in neighbor~positions.
		
		{Now suppose} that $a_{22} > 0$.  Let $\hat x_1 = \bar x_1$,
		$$\hat x_2 =\bar x_2 + {\bar c_{y2} \over 2\bar a_{22}}\bar x_1,$$			$\hat y_1 = \bar y_1$, $\hat y_2 = \bar y_2$			and $h(\hat \x, \hat \y) \equiv g(\bar \x, \bar\y)$. Then $h$ has only two half-cross terms, which are in neighbor positions, i.e.,
		\begin{align*}
			h(\hat \x, \hat \y)	= &	g(\bar \x, \bar \y)  \\
			= & \left(\bar a_{11}+\frac{\bar a_{21}\bar c_{y2}^2}{4\bar a_{22}^2}-\frac{\bar c_{y1}\bar c_{y2}}{2\bar a_{22}}\right)\hat x_1^2\hat y_1^2 + \left(\bar a_{12}-\frac{\bar c_{y2}^2}{4\bar a_{22}}\right)\hat x_1^2\hat y_2^2+\bar a_{21}\hat x_2^2\hat y_1^2+\bar a_{22}\hat x_2^2\hat y_2^2\\
			&+\bar b\hat x_1\hat x_2\hat y_1\hat y_2+\left(\bar c_{x1}-{\frac{b\bar c_{y2}}{2\bar a_{22}}}\right)\hat x_1^2\hat y_1\hat y_2+\left(\bar c_{y1}-\frac{\bar a_{21}\bar c_{y2}}{\bar a_{22}}\right)\hat x_1\hat x_2\hat y_1^2,
		\end{align*}
		which is in the form of \eqref{functionh}.
		This completes the proof.
	\end{proof}
	
	Combining Propositions~\ref{prop:3.1} and \ref{prop:3.2}, we obtain  the following~theorem.
	
	\begin{Thm}
		Suppose that the $2\times 2$ biquadratic polynomial $f$ expressed by \eqref{equ:2x2_BQ} is {PSD.} Then
		$a_{11}, a_{12}, a_{21}, a_{22} \ge 0$, and~\begin{equation}
			4a_{11}a_{12} \ge c_{x1}^2,\ 4a_{11}a_{21} \ge c_{y1}^2,\ 4a_{12}a_{22} \ge c_{y2}^2,\ 4a_{21}a_{22} \ge c_{x2}^2.
		\end{equation}
		{Furthermore},
		there is a $2\times 2$ {PSD} biquadratic polynomial $h$ with only two half-cross terms as expressed by {(\ref{functionh}),}
		such that if $h$ has an {SOS} form, then the {SOS} form of $f$ can be resulted.
	\end{Thm}
	
	{By Theorem~\ref{BQ-TRI}, any $2\times 2$ PSD biquadratic polynomial admits an SOS decomposition.} 
	Thus, in~the remaining part of this section, we consider a $2\times 2$ {PSD} biquadratic polynomial $f$, with~the form
	\begin{align}
		\nonumber	f(\x, \y) = & a_{11}x_1^2y_1^2 + a_{12}x_1^2y_2^2+a_{21}x_2^2y_1^2+a_{22}x_2^2y_2^2+bx_1x_2y_1y_2\\
		&+c_{x1}x_1^2y_1y_2{+c_{y1}}x_1x_2y_1^2, \label{equ:2x2_BQ_2halfcrossterms}
	\end{align}
	to identify its concrete {SOS} form.   In~Sections~\ref{sec4.1}--\ref{sec4.3}, we give constructive proofs for the {SOS} form of a $2 \times 2$ {PSD} biquadratic polynomials, in~three cases: I. $f$ has no half-cross terms; II. $f$ has only one half-cross term; III. $f$ has  two neighbor half-cross terms and no full-cross terms.
	{The case that the $2 \times 2$ bipartite polynomial with one full-cross term along with two neighbor  half-cross terms is somewhat complicated.  We do not consider this case~here.}

	\subsection{Case I: No Half-Cross~Terms\label{sec4.1}}
	
	If $f$ has no cross terms at all, the~case is trivial.   In~this subsection, we   assume that $f$ has one full-cross term but no half-cross~terms.
	
	\begin{Lem}\label{lem:bdmt_psd_sos}
		Let
		$$f(\x, \y) = a_{11} x_1^2y_1^2 + a_{12} x_1^2y_2^2 + a_{21}  x_2^2y_1^2 + a_{22} x_2^2y_2^2 - 4x_1x_2y_1y_2.$$
		{Then}
		the following three statements are equivalent:
		\begin{enumerate}
			\item[(a)] $f(\x, \y)$ is {PSD.}
			
			\item[(b)] $f(\x, \y)$ is  {SOS.}
			
			\item[(c)]   The parameters $a_{11}, a_{12}, a_{21}, a_{22}$ satisfy
			$$a_{11}, a_{12}, a_{21}, a_{22}  \ge 0 \text{ and } \sqrt{a_{11}a_{22}} + \sqrt{a_{12}a_{21}} \ge 2.$$\end{enumerate}
	\end{Lem}
	{
		\begin{proof}
			$(a) \Rightarrow (c)$ Since $f$ is {PSD} when $x_1=y_1=1$  and $x_2=y_2=0$, we obtain  $a_{11}\ge0$. Similarly,   $a_{12},a_{21},a_{22}\ge0$.
			If  $\sqrt{a_{11}a_{22}} \ge 2$, then
			\begin{eqnarray*}
				f(\vx,\vy) = \left(a_{11}-\frac{4}{a_{22}}\right)x_1^2y_1^2+ a_{12} x_1^2y_2^2 + a_{21}  x_2^2y_1^2 +\left(\frac{2}{\sqrt{a_{22}}}x_1y_1-\sqrt{a_{22}}x_2y_2\right)^2.
			\end{eqnarray*}
			{{By}
				analogous reasoning, if}  $\sqrt{a_{12}a_{21}} \ge 2$, then
			\begin{eqnarray*}
				f(\vx,\vy) = a_{11}x_1^2y_1^2+\left(a_{12}-\frac{4}{a_{21}}\right)x_1^2y_2^2+\left(\frac{2}{\sqrt{a_{21}}}x_2y_1-\sqrt{a_{21}}x_2y_1\right)^2+a_{22}x_2^2y_2^2.
			\end{eqnarray*}
			
			Now we assume that $\sqrt{a_{11}a_{22}} < 2$ and $\sqrt{a_{12}a_{21}} < 2$. {It follows from $f$ is {PSD} when $x_2=y_2=1$, $\sqrt{a_{12}}x_1y_2=\sqrt{a_{21}}x_2y_1$, and~$x_1,y_1\rightarrow \infty$ that} $a_{11} > 0$.  {Following similar methodology, we obtain that} $a_{12},a_{21},a_{22}>0$.
			By direct computation, we  derive that
			\begin{eqnarray*}
				f(\vx,\vy)&=&\left(\sqrt{a_{11}}x_1y_2-\sqrt{a_{22}}x_2y_1\right)^2+\left(\frac{2-\sqrt{a_{11}a_{22}}}{\sqrt{a_{21}}}x_1y_2-\sqrt{a_{21}}x_2y_1\right)^2\\
				&&+ \left(a_{12}-\frac{(2-\sqrt{a_{11}a_{22}})^2}{a_{21}}\right)x_1^2y_2^2.
			\end{eqnarray*}
			{{When}
				the first two terms vanish, the~PSD  condition of $f$  implies}   $\sqrt{a_{11}a_{22}} + \sqrt{a_{12}a_{21}} \ge 2$.
			
			$(c) \Rightarrow (b)$ follows directly from the above three reformulations of $f$.
			
			$(b) \Rightarrow (a)$ is obvious.
		\end{proof}
	}

	Building upon Lemma~\ref{lem:bdmt_psd_sos}, we now establish the following~theorem.
	
	\begin{Thm}\label{cor:bdmt_psd_sos}
		Let
		$$f(\x, \y) = a_{11} x_1^2y_1^2 + a_{12} x_1^2y_2^2 + a_{21}  x_2^2y_1^2 + a_{22} x_2^2y_2^2 - {b} x_1x_2y_1y_2.$$
		{Then,} 		the following three statements are equivalent:
		\begin{enumerate}
			\item[(a)] $f(\x, \y)$ is {PSD.}
			
			\item[(b)] $f(\x, \y)$ is  {SOS.}
			
			\item[(c)]   The parameters $a_{11}, a_{12}, a_{21}, a_{22}$ satisfy
			\begin{equation}
				a_{11}, a_{12}, a_{21}, a_{22}  \ge 0, \text{ and } \sqrt{a_{11}a_{22}} + \sqrt{a_{12}a_{21}} \ge {|{b}| \over 2}.
		\end{equation}\end{enumerate}

	\end{Thm}

	\subsection{Case II: Only One Half-Cross~Term\label{sec4.2}}
	
	In this subsection, we assume that $f$ has {one} half-cross term and one full-cross~term.
	
	\begin{Lem}\label{Lem:psd_2tail}
		Suppose that
		$$f(\x, \y) =  {a_{11} x_1^2y_1^2 + a_{12} x_1^2y_2^2 +   a_{21} x_2^2y_1^2+ a_{22} x_2^2y_2^2
			- 4 x_1y_1x_2y_2 - c_{y1} x_1x_2y_1^2.}$$
		{Then,} the following three statements are equivalent:
		\begin{enumerate}
			\item[(a)] $f(\x, \y)$ is {PSD.}
			
			\item[(b)] $f(\x, \y)$ is {SOS.}
			
			\item[(c)] There exist {$a_{111}, a_{211}\ge0$ and $a_{112}, a_{212}> 0$} such that
			$a_{11} = a_{111} + a_{112}$, \mbox{$a_{21} = a_{211} + a_{212}$},   $4a_{112} a_{212}=c_{y1}^2$, and~		$\sqrt{a_{111}a_{22}} + \sqrt{a_{12}a_{211}} \ge 2$.\end{enumerate}
	\end{Lem}
	\begin{proof} If $c_{y1} = 0$, then the conclusion follows from Lemma \ref{lem:bdmt_psd_sos}.  Thus, we may assume that $c_{y1} \not = 0$.

		(a) $\Rightarrow$ (c) Because $f$ is {PSD} when $y_2 = 0$, we obtain that  $4a_{11} a_{21} \ge c_{y1}^2$.   Therefore, there exist {$a_{111}, a_{211}\ge0$ and $a_{112}, a_{212}> 0$} such that $a_{11} = a_{111} + a_{112}$, \mbox{$a_{21} = a_{211} + a_{212}$},   and~$4a_{112}a_{212} = c_{y1}^2$.
		In the following, we show $\sqrt{a_{111}a_{22}} + \sqrt{a_{12}a_{211}} \ge 2$.
		Let \mbox{$x_1 = \sqrt{a_{212}}>0$} and $x_2 = {\rm sgn}(c_{y1})\sqrt{a_{112}}\neq 0$.   Then,
		$${f}(\x, \y) =   \left(a_{111} a_{212} + a_{211} a_{112}\right)y_1^2 + \left(a_{22}a_{112}+a_{12}a_{212}\right) y_2^2
		-4 {\rm sgn}(c_{y1})\sqrt{a_{112}a_{212}}y_1y_2.$$
		{Since}
		$f$ is {PSD,}
		we see that $a_{111}$ and $a_{211}$ cannot be both zero, and~		$a_{22}$ and $a_{12}$ cannot be both zero.
		We consequently classify three scenarios based on  the parameters  $a_{22}$ and $a_{12}$:
		
		Case (i) $a_{22}=0$ {and $a_{12}>0$}.
		It follows from $f$ is {PSD} when $x_1=\sqrt{a_{21}}$, $x_2=\sqrt{a_{11}}$, $y_2=1$, and~$y_1\rightarrow \infty$ that  $a_{21}> {c_{y1}^2 \over 4a_{11}}$.
		By selecting the parameters as $a_{112} = a_{11}$, $a_{212} = {c_{y1}^2 \over 4a_{11}}$, $a_{211} = a_{21} - a_{212}$,   we obtain $a_{211}>0$ and
		\begin{eqnarray*}
			{f}(\x, \y) &=&
			\left(a_{12}-\frac{4}{a_{211}}\right)x_1^2y_2^2 +  \left(\frac{2}{\sqrt{a_{211}}}x_1y_2-\sqrt{a_{211}}x_2y_1\right)^2 \\
			&& + \left(\sqrt{a_{11}}x_1-{\rm sgn}(c_{y1})\sqrt{a_{212}}x_2 \right)^2y_1^2.
		\end{eqnarray*} 		
		{Let} $x_1 = \sqrt{a_{212}}$, $x_2 = {\rm sgn}(c_{y1})\sqrt{a_{112}}$, $y_2 = \sqrt{a_{211}}x_2$ and $y_1 = \sqrt{a_{12}}x_1$.  Then \mbox{$f(\x, \y) = 	 \left(a_{12}-\frac{4}{a_{211}}\right)x_1^2y_2^2$}. It follows from  $f$ is {PSD} that  $\sqrt{a_{12}a_{211}} \ge 2$.
		Consequently, we conclude that there  exist  parameters $a_{111}, a_{211}, a_{112}, a_{212}$    satisfying condition (c).

		Case (ii) $a_{12}=0$ {and $a_{22}>0$}.  This can be proved similarly as Case (i).
		
		Case (iii) $a_{22}>0$ and $a_{12} > 0$.
		Note that $a_{111}$ and $a_{211}$ cannot both be  zero.  Suppose  $a_{111} \not = 0$. {If   $\sqrt{a_{111}a_{22}}>2$, condition (c) hold directly.  Suppose that   $\sqrt{a_{111}a_{22}}<2$.} Then it follows from $f$ is {PSD} that $a_{211}>0$ and\vspace{-3pt}
		\begin{align*}
			{f}(\x, \y) = &
			(\sqrt{a_{111}}x_1y_1-\sqrt{a_{22}}x_2y_2)^2 +\left(a_{12}-\frac{(2-\sqrt{a_{111}a_{22}})^2}{a_{211}}\right)x_1^2y_2^2\\
			&+\left(\frac{2-\sqrt{a_{111}a_{22}}}{\sqrt{a_{211}}}x_1y_2-\sqrt{a_{211}}x_2y_1\right)^2 \\
			& + \left(\sqrt{a_{112}}x_1-{\rm sgn}(c_{y1})\sqrt{a_{212}}x_2 \right)^2y_1^2.
		\end{align*}
		
		Let  $x_1 = \sqrt{a_{212}}$, $x_2 = {\rm sgn}(c_{y1})\sqrt{a_{112}}$, $y_1 = \sqrt{a_{22}}x_2$ and $y_2 = \sqrt{a_{111}}x_1$.  	{Then the first two terms are zeros. Furthermore, define}
		$$\phi(a_{112}) = \sqrt{a_{111}}(2-\sqrt{a_{111}a_{22}})c_{y1}^2 - 4a_{211}a_{112}^2\sqrt{a_{22}}.$$ {It} follows from  $a_{111} = a_{11} - a_{112}$, $a_{211} =  a_{21} - a_{212}$, and~$a_{212} = {c_{y1}^2 \over 4a_{112}}$ that $\phi$ is a function of $a_{112}$.
		Since  $a_{212} = {c_{y1}^2 \over 4a_{112}}\le a_{21}$, it yields that  $a_{112}\in \left[\frac{c_{y1}^2}{4a_{21}}, a_{11}\right]$.   When  $a_{112} = \frac{c_{y1}^2}{4a_{21}}$, we derive that  $a_{211} = 0$ and $\phi\left(\frac{c_{y1}^2}{4a_{21}}\right) \ge 0$.
		When $a_{112} = a_{11}$,  we obtain that  $a_{111} = 0$ and $\phi(a_{11}) \le 0$.   Thus, there is $\eta_0 \in {\left[\frac{c_{y1}^2}{4a_{21}}, a_{11}\right]}$ such that $\phi(a_{112}) = 0$. Then by setting $a_{112} = \eta_0$, $a_{212} = {c_{y1}^2 \over 4\eta_0}$, we obtain that $f(\x, \y)= \left(a_{12}-\frac{(2-\sqrt{a_{111}a_{22}})^2}{a_{211}}\right)x_1^2y_2^2$. It follows from $f$ is {PSD} that
		$\sqrt{a_{111}a_{22}}+\sqrt{a_{12}a_{211}} \ge 2$.
		Consequently, we derive that   there  exist  parameters $a_{111}, a_{211}, a_{112}, a_{212}$, specifically chosen as $a_{112} = \eta_0$, $a_{212} = {c_{y1}^2 \over 4\eta_0}$, that satisfy (c).

		{The case when $a_{211} \not = 0$ follows by an analogous argument with $a_{111} \not = 0$.}
		
		(c) $\Rightarrow$ (b) follows directly from the above  reformulations of $f$.

		(b) $\Rightarrow$ (a)    follows directly from the~definitions.
		
		This completes the proof.
	\end{proof}
	
	{Furthermore}, we establish an equivalent condition based on  Lemma~\ref{Lem:psd_2tail} that depends solely on the coefficients of $f$.
	\begin{Thm}
		Suppose that
		$$f(\x, \y) = {a_{11}} x_1^2y_1^2 + a_{22} x_2^2y_2^2 + {a_{12}}  x_1^2y_2^2 + {a_{21}} x_2^2y_1^2
		- b x_1y_1x_2y_2 - c_{y1} x_1x_2y_1^2.$$
		{Let}
		\[g(a_{112})=\sqrt{\left(a_{11}- a_{112}\right)a_{22}} + \sqrt{a_{12}\left(a_{21}-\frac{c_{y1}^2}{4a_{112}}\right)}.\]
		and $a_{112}^*\in \left(\frac{c_{y1}^2}{4a_{21}}, a_{11}\right)$ be the unique parameter such that {the derivative} $g'(a_{112}^*)=0$.
		Then the following three statements are equivalent:
		\begin{enumerate}
			\item[(a)] $f(\x, \y)$ is {PSD.}
			
			\item[(b)] $f(\x, \y)$ is {SOS.}
			
			\item[(c)] It holds that
			$g(a_{112}^*)\ge {\frac{|b|}{2}}$.\end{enumerate}
	\end{Thm}
	\begin{proof}
		Based on Lemma~\ref{Lem:psd_2tail}, we obtain that   $a_{11} = a_{111} + a_{112}$, $a_{21} = a_{211} + a_{212}$,   \mbox{$4a_{112}a_{212} = c_{y1}^2$}, and~		\begin{eqnarray*}
			g(a_{112})=	 \sqrt{\left(a_{11}- a_{112}\right)a_{22}} + \sqrt{a_{12}\left(a_{21}-\frac{c_{y1}^2}{4a_{112}}\right)}=\sqrt{a_{111}a_{22}} + \sqrt{a_{12}a_{211}}.
		\end{eqnarray*}
		{For} any  $a_{112}\in \left(\frac{c_{y1}^2}{4a_{21}}, a_{11}\right)$, we may derive {the derivative} of $g$ as follows:
		\begin{equation*}
			g'(a_{112}) = -\frac{\sqrt{a_{22}}}{2\sqrt{a_{11}- a_{112}}} + \frac{\sqrt{a_{12}}\frac{c_{y1}^2}{4a_{112}^2}}{2\sqrt{a_{21}-\frac{c_{y1}^2}{4a_{112}}}}.
		\end{equation*}
		{Here}, the~first term $\frac{\sqrt{a_{22}}}{2\sqrt{a_{11}- a_{112}}}$ is monotonously increasing to infinity and the second  term $\frac{\sqrt{a_{12}}\frac{c_{y1}^2}{4a_{112}^2}}{2\sqrt{a_{21}-\frac{c_{y1}^2}{4a_{112}}}}$ is monotonously decreasing from  infinity. Therefore, there is one unique parameter $a_{112}^*\in \left(\frac{c_{y1}^2}{4a_{21}}, a_{11}\right)$ such that $g'(a_{112}^*)=0$. Furthermore, $g(a_{112})$ is monotonously increasing in $a_{112}\in \left(\frac{c_{y1}^2}{4a_{21}}, a_{112}^*\right)$ and then  monotonously {decreasing} in $a_{112}\in \left(a_{112}^*, a_{11}\right)$.
		Consequently, $g(a_{112}^*) = \max_{a_{112}\in \left[\frac{c_{y1}^2}{4a_{21}}, a_{11}\right]} g(a_{112})$.
		Therefore, if~$f$ is {PSD,} then it holds that $g(a_{112}^*)\ge \frac{|b|}{2}$.
	\end{proof}

	\subsection{Case III: Two Neighbor Half-Cross Terms and No Full-Cross~Terms\label{sec4.3}}
	
	In this subsection, we study the case of $2 \times 2$ biquadratic polynomials with  two half-cross terms  in the neighbor sides and no full-cross term.  The~case that the two half-cross terms in the opposite sides and no full-cross term is somewhat trivial.  Consider
	\begin{equation} \label{NewCase-1}
		f(\x, \y) = a_{11}x_1^2y_1^2 + a_{12}x_1^2y_2^2+a_{21}x_2^2y_1^2+a_{22}x_2^2y_2^2
		+2c_xx_1^2y_1y_2+{2c_y}x_1x_2y_1^2.
	\end{equation}
	{Here}, we denote $c_x = {c_{x1} \over 2}$ and $c_y = {c_{y1} \over 2}$ to simplify our notation.
	Suppose that $c_x \not = 0$ and $c_y \not = 0$, and~$f$ is {PSD.}   Then we derive that
	$$a_{11}a_{12} \ge c_{x}^2 > 0,\ a_{11}a_{21} \ge c_{y}^2 > 0,$$
	$a_{11}, a_{12}, a_{21} > 0$ and $a_{22} \ge 0$.
	
	For simplicity, we may replace $x_2$ and $y_2$ by ${\bar x_2 \over \sqrt{a_{21}}}$ and ${\bar y_2 \over \sqrt{a_{12}}}$, respectively, and~assume that $a_{12} = a_{21} = 1$, i.e.,
	\begin{equation} \label{NewCase-1-1}
		f(\x, \y) = a_{11}x_1^2y_1^2 + x_1^2y_2^2+x_2^2y_1^2+a_{22}x_2^2y_2^2
		+2c_{x}x_1^2y_1y_2+2c_{y}x_1x_2y_1^2.
	\end{equation}
	{We} also may assume $c_x > 0$ (resp. $c_y > 0$), since otherwise we could replace $y_1$ with $-y_1$ (resp. $x_1$ with $-x_1$).
	Without loss of generality, assume that {$c_{x} \ge c_{y}>0$.}
	
	If $a_{11} \ge {c_{x}^2} + {c_{y}^2}$, then
	\begin{align*}
		f(\x, \y) = &  \left(a_{11}-{c_{x}^2} - {c_{y}^2}\right)x_1^2y_1^2 + x_1^2\left(c_{x} y_1+y_2\right)^2
		+ \left(c_yx_1+x_2\right)^2y_1^2 +a_{22}x_2^2y_2^2,
	\end{align*}
	which gives the {SOS} form of $f$ explicitly.

	When $a_{11} =  c_{x}^2$,
	$f(\x, \y)$ is not {PSD} for any $a_{22} > 0$.   This can be seen by letting $y_1 = 1$, $y_2 = -c_{x}$, $x_1 = - c_y$ and
	$$ 0 < x_2 < {2c_{y}^2 \over 1 + {a_{22}c_{x}^2}}.$$
	
	\begin{Thm}\label{Thm:case3}
		Suppose that the $2 \times 2$ biquadratic polynomial  $f$  is expressed by~(\ref{NewCase-1-1}), where \mbox{${c_{x}^2} < a_{11} < {c_{x}^2} + {c_{y}^2}$}, $a_{22}>0$, and~$c_x\ge c_y>0$.
		Then the following three statements are equivalent:
		\begin{enumerate}
			\item[(a)] $f(\x, \y)$ is {PSD.}
			
			\item[(b)] $f(\x, \y)$ is {SOS.}
			
			\item[(c)]   It holds that
			\[a_{22}\ge \begin{cases}
				\frac{8c_y}{\left(3c_y-\sqrt{9c_y^2-4a_{11}}\right)^3}-\frac{4}{\left(3c_y-\sqrt{9c_y^2-4a_{11}}\right)^2},  & \text{ if } c_x=c_y \text{ and } \frac54 c_y^2\le a_{11}\le 2c_y^2,\\  	 \frac{1}{a_{11}-c_x^2}-\frac{4}{\left(3c_y-\sqrt{9c_y^2-4a_{11}}\right)^2},  & \text{ if } c_x=c_y \text{ and }  c_y^2\le a_{11}\le \frac54c_y^2,\\
				\frac{(\gamma_1^*)^2(1-\gamma_1^*)(2-\gamma_1^*)^2}{((\gamma_1^*-1)c_x+c_y)(c_x+(\gamma_1^*-1)c_y)}, & \text{ if } c_x>c_y,
			\end{cases}
			\]
			where $\gamma_1^*\in \left(\frac{c_x-c_y}{c_x},1\right)$ is the root  of  function $\Omega$ defined by the following equation
		\begin{equation}\label{equ:Omega}
					\Omega(\gamma_1):=a_{11}\gamma_1^2(2-\gamma_1)^2-3(1-\gamma_1)^3c_xc_y+2(1-\gamma_1)^2(c_x^2+c_y^2)+(1-\gamma_1)c_xc_y-c_x^2-c_y^2=0.
		\end{equation}
	\end{enumerate}
	\end{Thm}
	\begin{proof}
		We assume that $f$ has the following {SOS} decomposition:
		\begin{align}
			\nonumber f(\x, \y) = & \gamma_1 (\alpha x_1 + x_2)^2y_1^2 + \gamma_1 x_1^2(\beta y_1+y_2)^2 + \gamma_2\left[(\alpha+\beta) x_1y_1 + x_1y_2+ x_2y_1\right]^2\\
			& + \gamma_3(\alpha\beta x_1y_1 - x_2y_2)^2 + \left(a_{22} -\gamma_3\right)x_2^2y_2^2, \label{CaseII-1-sos}
		\end{align}
		where parameters $\gamma_1, \gamma_2\ge 0$ and $0 \le \gamma_3 \le a_{22}$. By~comparing~(\ref{NewCase-1-1}) and~(\ref{CaseII-1-sos}),  we observe~that\vspace{-18pt}
		\begin{subequations} \label{equ:abc}
			\begin{align}
				\gamma_1(\alpha^2+\beta^2) + \gamma_2(\alpha+\beta)^2 +\gamma_3\alpha^2\beta^2 & = a_{11},\label{pa1}\\
				\gamma_1 + \gamma_2 & = 1,\label{pa2}\\
				\gamma_1\beta + \gamma_2(\alpha+\beta) & = c_x,\label{pa3}\\
				\gamma_1\alpha + \gamma_2(\alpha+\beta) & = c_y,\label{pa4}\\
				\gamma_2 - \gamma_3\alpha\beta & = 0.\label{pa5}
			\end{align}
		\end{subequations}
		{{The} polynomial $f$ admits an SOS decomposition if the system of Equation \eqref{equ:abc}  admits solutions $\alpha, \beta, \gamma_1, \gamma_2, \gamma_3$  satisfying $\gamma_1, \gamma_2\ge 0$ and $0 \le \gamma_3 \le a_{22}$.}
		Combining \mbox{Equations~(\ref{pa3}) and~(\ref{pa4})} yields
		$(\alpha - \beta)\gamma_1 = c_y - c_x$.
		In the   case where $c_y=c_x$,    either $\alpha = \beta$ or $\gamma_1 = 0$.  
		{We therefore proceed by analyzing two distinct cases separately:
			(i)~$c_y = c_x$ and    (ii) $c_y \neq c_x$.
			The former case is subsequently divided into subcases (i1) and (i2) for detailed examination.}

		Case (i1): $c_y = c_x$ and $\frac54 c_y^2\le a_{11}\le 2c_y^2$.   By~setting $\alpha = \beta{\neq 0}$ and {$\gamma_2 = 1-\gamma_1$, we can simplify Equations~(\ref{pa1})--(\ref{pa5}) to}
		\begin{subequations}
			\begin{align}
				{2\gamma_1\alpha^2+ 4(1-\gamma_1)\alpha^2} +\gamma_3\alpha^4 & = a_{11},\label{pa8}\\
				\gamma_1\alpha + 2(1-\gamma_1)\alpha & = c_y,\label{pa9}\\
				1-\gamma_1 - \gamma_3\alpha^2 & = 0.\label{pa10}
			\end{align}
		\end{subequations}
		{Combining}
		{equations}~(\ref{pa8}) and~(\ref{pa10}) yields
		\begin{align*}
			\alpha^2(5-3\gamma_1) = a_{11} \text{ and }
			(2-\gamma_1) \alpha =c_y.
		\end{align*}
		{Therefore}, we obtain that  $\gamma_1 = 2-\frac{c_y}{\alpha}$ and $\alpha^2-3c_y\alpha+a_{11}=0$. It follows from  $a_{11} \le 2c_y^2$  that the quadratic equation in $\alpha$ admits two distinct real roots  $\alpha=\frac{3c_y\pm\sqrt{9c_y^2-4a_{11}}}{2}$.  By~the fact that $0\le \gamma_1 = 2-\frac{c_y}{\alpha}\le 1$, only the  root
		$\alpha^*=\frac{3c_y-\sqrt{9c_y^2-4a_{11}}}{2}$ is feasible.
		Consequently, in~Case (i1), if~		$$a_{22} \ge \gamma_3^* = {c_y\over (\alpha^*)^4}-{1\over (\alpha^*)^2},$$
		then $f$ is {SOS.}

		Case (i2): $c_y = c_x$ and $c_y^2<a_{11}\le\frac54 c_y^2$. By~setting $\gamma_1=0$ and $\gamma_2 = 1$,  we can simplify~(\ref{pa1})--(\ref{pa5}) to
		\begin{align}\label{equ:case(i2)}
			(\alpha+\beta)^2+\alpha\beta = a_{11}, \ \  \alpha+\beta=c_y, \  \text{ and } \gamma_3=\frac{1}{\alpha\beta}.
		\end{align}
		{By} direct computation, we   obtain $\beta = c_y-\alpha$, $c_y^2+\alpha\beta=a_{11}$, and~ $\gamma_3=\frac{1}{a_{11}-c_y^2}$. Consequently, we  derive the following quadratic equation with respective to $\alpha$:
		\begin{align}\label{equ:alpha}
			\alpha^2-c_y\alpha+a_{11}-c_y^2=0.
		\end{align}
		{It} follows from $c_y^2<a_{11} \le \frac54c_y^2$  that  \eqref{equ:alpha} admits two real solutions, which derives that \eqref{equ:case(i2)} is feasible.
		Hence, in~Case (i2), if~$$a_{22} \ge \gamma_3^* = \frac{1}{a_{11}-c_x^2},$$ then $f$ is {SOS.}

		
		Case (ii): $c_x> c_y$
		and $c_x^2<a_{11}<c_x^2+c_y^2$. Then we obtain that $\alpha\neq \beta$, and~$\gamma_1\neq 0$.
		It follows from~(\ref{pa2})--(\ref{pa5})   that
		\begin{align}\label{equ:gamma23albe}
			\gamma_2=1-\gamma_1, \ \gamma_3 = \frac{1-\gamma_1}{\alpha\beta}, \ \alpha=\frac{(\gamma_1-1)c_x+c_y}{\gamma_1(2-\gamma_1)},  \ \beta=\frac{c_x+(\gamma_1-1)c_y}{\gamma_1(2-\gamma_1)}.
		\end{align}
		{Consequently},   Equation~(\ref{equ:Omega}) is derived directly from \eqref{pa1}.
		By the inequalities\linebreak $\Omega(1) = a_{11}-c_x^2-c_y^2 <0$ and
		$\Omega\left(\frac{c_x-c_y}{c_x}\right)=\frac{(c_x^2-c_y^2)^2}{c_x^4}(a_{11}-c_x^2)>0$, we conclude that there exists a solution $\gamma_1^*\in \left(\frac{c_x-c_y}{c_x},1\right)$ such that $\Omega(\gamma_1^*)=0$.
		Subsequently, we obtain the corresponding values of $\gamma_2^*, \gamma_3^*, \alpha^*$ and $\beta^*$ by \eqref{equ:gamma23albe}.
		Furthermore, it follows directly    from $c_x> c_y$  and $\gamma_1^*\in (0,1)$ that $\beta^*>0$, and~from    $\gamma_1^*>\frac{c_x-c_y}{c_x}$ that
		$\alpha^*>0$.  Consequently,  $\gamma_3^*=\frac{1-\gamma_1^*}{\alpha^*\beta^*}> 0$. Thus, if~$$a_{22} \ge \gamma_3^*=	 \frac{(\gamma_1^*)^2(1-\gamma_1^*)(2-\gamma_1^*)^2}{((\gamma_1^*-1)c_x+c_y)(c_x+(\gamma_1^*-1)c_y)}>0,$$
		then $f$ is {SOS.}
		
		
		In both Cases (i) and   (ii), if~$a_{22} < \gamma_3^*$, then $f$ is not {PSD.}   This can be seen by letting $x_2 = - \alpha x_1 = 1$ and $y_2 = -\beta y_1 = 1$ in~(\ref{CaseII-1-sos}).
		This completes the proof.
	\end{proof}

	\begin{example}
			Let $a_{11}=1.2$, $a_{22}=6$, $c_x=c_y=-1$. The~corresponding biquadratic polynomial is given by:
			\begin{eqnarray*}
				f(\vx,\vy) = 1.2x_1^2y_1^2+x_1^2y_2^2+x_2^2y_1^2+6x_2^2y_2^2-2x_1^2y_1y_2-2x_1x_2y_1^2.
			\end{eqnarray*}
			{We} could verify that $c_y^2\le a_{11}\le \frac54c_y^2$  and $a_{22}=6>\frac{1}{a_{11}-c_x^2}-\frac{4}{\left(3c_y-\sqrt{9c_y^2-4a_{11}}\right)^2}=4.8431$. According to Theorem~\ref{Thm:case3}, these conditions guarantee that $f$ is an  SOS. By~  YALMIP   and SDPT4  within the MATLAB environment,  we obtain an explicit SOS decomposition as follows:\vspace{-3pt}
			\begin{eqnarray*}
				f(\vx,\vy)&=&(-0.4876x_1y_1+0.1140x_1y_2+0.1140x_2y_1+2.4326x_2y_2)^2\\
				&&+(0.9781x_1y_1-0.9685x_1y_2-0.9685x_2y_1+0.2869x_2y_2)^2\\
				&&+ (0.2181x_1y_2-0.2181x_2y_1)^2\\
				&&+(0.0739x_1y_1+0.0390x_1y_2+0.0390x_2y_1+0.0112x_2y_2)^2.
			\end{eqnarray*}
		\end{example}
		
		\section{Conclusions and Open~Questions}\label{sec5}
	In his seminal 1888 work~\cite{Hi88}, Hilbert demonstrated the existence of  {PSD}  homogeneous quartic polynomials in four variables that cannot be expressed as {SOS.} Nearly a century later, Choi and Lam (1977) \cite{CL77} constructed an explicit example of such a  {PNS} quartic polynomial. In~the present work, we establish that despite being four-variable quartic polynomials, all $2\times2$ {PSD} biquadratic polynomials admit an {SOS} representation.
	The key technique is that  an $m \times n$ biquadratic polynomial {can} be expressed as a tripartite homogeneous quartic polynomial of $m+n-1$ variables. {Furthermore, we established a necessary and sufficient condition for an $m\times n$ PSD biquadratic polynomial to admit an SOS representation, and~demonstrated that the SOS rank of such polynomials was bounded above by $mn$. Subsequently, we provided a constructive proof for obtaining the SOS decomposition of $2\times 2$ PSD biquadratic polynomials in three specific cases.}

	In general, an~$m \times n$ biquadratic polynomial can be written as
	$$f(\x, \y) = \sum_{i, j=1}^m \sum_{k, l = 1}^n a_{ijkl}x_ix_jy_ky_l.$$
	{In 1973, Calder\'{o}n \cite{Ca73} proved that an $m \times 2$ PSD biquadratic polynomial can be expressed as the sum of squares of ${3m(m+1) \over 2}$ quadratic polynomials.   In this paper, we
		proved that a   $2 \times 2$ PSD biquadratic polynomial can be expressed as the sum of squares of three quadratic polynomials.    An open question is what is the SOS rank for an $m \times 2$ PSD biquaratic polynomial for $m \ge 3$.     In 1975, Choi \cite{Ch75} gave a $3 \times 3$ PNS biquadratic polynomial.   Another
		open question is:  What is the maximum SOS rank of an $m \times n$ SOS biquadratic polynomial for $m, n \ge 3$.    An even more challenging question is to determine a given $m \times n$ PSD biquadratic polynomial is SOS or not numerically.}

	{{\bf Acknowledgment}}
	This work was partially supported by Research  Center for Intelligent Operations Research, The Hong Kong Polytechnic University (4-ZZT8),    the National Natural Science Foundation of China (Nos. 12471282 and 12131004), the R\&D project of Pazhou Lab (Huangpu) (Grant no. 2023K0603),  the Fundamental Research Funds for the Central Universities (Grant No. YWF-22-T-204), and Jiangsu Provincial Scientific Research Center of Applied Mathematics (Grant No. BK20233002).

	{{\bf Data availability}
		No datasets were generated or analysed during the current study.

		{\bf Conflict of interest} The authors declare no conflict of interest.}

	


\end{document}